\documentclass[12pt]{amsart}
\usepackage{amsmath}
\usepackage[backref=page]{hyperref}
\usepackage{tikz,float,caption}

\usepackage[margin=1.25 in,top=1.2in, bottom=1.4 in]{geometry}

\usepackage[bb=ams,scr=euler]{mathalpha}

\usepackage{setspace}
\usepackage{enumitem}
\setenumerate{
  wide=10pt,
  leftmargin=0pt,
  labelwidth=*,
  label=(\roman*),
  topsep=0pt,
  listparindent=0pt,
  parsep=4pt}

\usepackage{thmtools}
\declaretheoremstyle[headfont=\normalsize\normalfont\bfseries,notefont=\mdseries, notebraces={(}{)},bodyfont=\normalfont,postheadspace=0.5em]{basicstyle}
\declaretheoremstyle[headfont=\normalsize\normalfont\bfseries,notefont=\mdseries, notebraces={(}{)},bodyfont=\normalfont\itshape,postheadspace=0.5em]{italstyle}

\declaretheorem[name=Definition,style=basicstyle]{defn}
\declaretheorem[name=Remark,style=basicstyle,sibling=defn]{remark}

\declaretheorem[style=italstyle,name=Theorem]{theorem}
\declaretheorem[style=italstyle,name=Corollary,sibling=theorem]{cor}

\declaretheorem[style=italstyle,name=Proposition,sibling=theorem]{prop}

\declaretheorem[style=italstyle,name=Lemma,sibling=theorem]{lemma}
\declaretheorem[style=basicstyle,name=Proof,numbered=no,qed=$\square$]{preproof}
\renewenvironment{proof}{\preproof}{\endpreproof}
\newenvironment{dmatrix}{\left[\,\begin{matrix}}{\end{matrix}\,\right]}
\makeatletter
\newcommand{\abs}[1]{\left|#1\right|}
\newcommand{\bd}{\partial}

\newcommand{\C}{\mathbb{C}}
\renewcommand{\d}{\mathrm{d}}
\newcommand{\id}{\mathrm{id}}
\newcommand{\intprod}{\mathbin{{\tikz{\draw(-0.1,0)--(0.1,0)--(0.1,0.2)}\hspace{0.5mm}}}}

\newcommand{\R}{\mathbb{R}}
\newcommand{\set}[1]{\left\{#1\right\}}

\makeatother

\title{Contactomorphisms of the sphere without translated points}
\author{Dylan Cant}
\date{\today}
\begin{document}
\maketitle
\begin{abstract}
  We construct a contactomorphism of $(S^{2n-1},\alpha_{\mathrm{std}})$ which does not have any translated points, providing a negative answer to a conjecture posed in \cite{sandon_13}.
\end{abstract}

\section{Introduction}
\label{sec:intro}
Let $(Y^{2n-1},\alpha)$ be a contact manifold with a choice of contact form $\alpha$. Recall that this means that $\alpha$ is a $1$-form so that $\alpha\wedge \d\alpha^{n-1}$ is a volume form. A \emph{contactomorphism} is a diffeomorphism $\varphi:Y\to Y$ with the property that $\varphi^{*}\alpha=e^{g}\alpha$ for\footnote{Strictly speaking, our definition selects only the orientation preserving contactomorphisms.} some smooth function $g:Y\to \R$. The function $g$ is reasonably called the \emph{scaling factor}; indeed, we easily compute:
\begin{equation}\label{eq:volume}
  \varphi^{*}(\alpha\wedge \d\alpha^{n-1})=e^{ng}\alpha \wedge \d\alpha^{n-1},
\end{equation}
i.e., $e^{ng}$ governs the change in volume due to $\varphi$.

A choice of contact form also selects a special vector field $R$ called the \emph{Reeb field}, characterized by the equations $\alpha(R)=1$ and $\d\alpha(R,-)=0$.

We recall the following notion from \cite{sandon_12} and \cite{sandon_13}. Given a contactomorphism $\varphi$, a point $p\in Y$ is called a \emph{translated point} provided that $g(p)=0$ and $\varphi(p)$ lies on the Reeb flow line passing through $p$.

In \cite{sandon_13}, the author conjectures that every contactomorphism $\varphi$ isotopic to the identity of a compact contact manifold $Y$ (with any choice of form $\alpha$) has at least one translated point. The goal of the present document is to give counterexamples to this conjecture on $S^{2n-1}$ with the standard contact form $\alpha_{\mathrm{std}}$, for $n>1$. The main result we will prove is:

\begin{theorem}\label{theorem:main_result}
  Let $n>1$. There exist contactomorphisms $\varphi:S^{2n-1}\to S^{2n-1}$ isotopic to the identity which do not have translated points for the contact form $\alpha_{\mathrm{std}}$.
\end{theorem}
The proof is given in \S\ref{sec:proof} and \S\ref{sec:construction} below. We recall the definition of standard contact form $\alpha_{\mathrm{std}}$ in \S\ref{sec:sphere}.

\begin{remark}
  For the case $Y=(S^{1},\alpha)$, the identity $\int \varphi^{*}\alpha=\int \alpha$ implies the existence of at least two points satisfying $g=0$, and every point on $Y$ can be joined to any other by the Reeb flow. Thus every contactomorphism of $S^{1}$ has translated points (for any contact form).
\end{remark}

\begin{remark}
  Sandon's conjecture has been proved in multiple cases. In \cite{AFM15}, \cite{meiwes_naef}, it is shown if the contact form $\alpha$ is \emph{hypertight} in the sense that it admits no contractible Reeb orbits. In \cite{shelukhin_contactomorphism} (using \cite{leaf_wise_albers_frauenfelder}'s work on leaf-wise intersection points), and \cite{oh_legendrian_entanglement}, \cite{oh_shelukhin_2}, the existence of translated points is proved under a smallness assumption on the oscillation norm $\int_{0}^{1} \mathrm{max}(H_{t})-\mathrm{min}(H_{t})\d t$.\footnote{Here $H_{t}=\alpha(X_{t})$ is the contact Hamiltonian associated to the infinitesimal generator $X_{t}$ of a path of contactomorphisms $\varphi_{t}$} In \cite{albers_merry}, the authors prove Sandon's conjecture for the boundary of Liouville domains $X$ with non-vanishing Rabinowitz Floer homology, and \cite{merry_ulja} proves the conjecture when the symplectic homology of $X$ is infinite dimensional. The papers \cite{sandon_13}, \cite{gkps}, \cite{allais_lens} establish versions of Sandon's conjecture for (non-spherical) lens spaces ($\mathbb{RP}^{2n-1}$ is a special case), using a generating function approach and ideas from \cite{theret_rotation}, \cite{givental_quasimorphism}. This is further supported by the work of \cite{albers_kang}, where Rabinowitz Floer homology groups of lens spaces are shown to be non-zero. The work of \cite{allais_zoll} establishes Sandon's conjecture for certain unit cotangent bundles using a variational approach.
\end{remark}

\begin{remark}
  Sandon's work \cite{sandon_13} claimed to show that every contactomorphism on $(S^{2n-1},\alpha_{\mathrm{std}})$ had a translated point. However, \cite{gootjes_dreesbach} has clarified the situation by pointing out a gap in the original argument, and he gives a detailed proof of a restricted statement, see \cite[Theorem 1.2]{gootjes_dreesbach}. The author wishes to thank Sandon for introducing him to \cite{gootjes_dreesbach}.
\end{remark}

\begin{remark}\label{remark:special_case}
  Sandon's conjecture is straightforward to verify in the case when $\varphi_{t}$ is an autonomous family of contactomorphisms generated by a Hamiltonian $H$ which is constant on Reeb flow lines (i.e., $R\intprod \d H=0$). In this case, any critical point of $H$ is a translated point of $\varphi_{t}$, for all $t$.  
\end{remark}

\subsubsection*{Acknowledgements}
\label{sec:ack}
I want to thank Egor Shelukhin for introducing me to the problem of translated points, and for suggesting many improvements to the argument. This work was completed at the University of Montreal with funding from the CIRGET research group.

\section{The standard contact form on the sphere}
\label{sec:sphere}
Consider $S^{2n-1}$ as the unit sphere in $\R^{2n}$, and recall that
\begin{equation*}
  \alpha_{\mathrm{std}}=\sum_{i=1}^{n}(x_{i}\d y_{i}-y_{i}\d x_{i})
\end{equation*}
defines the standard contact form on $S^{2n-1}$. It is readily checkable that
\begin{equation*}
  R=\sum_{i=1}^{n}x_{i}\bd_{y_{i}}-y_{i}\bd_{x_{i}}=J\sum_{i=1}^{n}(x_{i}\bd_{x_{i}}+y_{i}\bd_{y_{i}})
\end{equation*}
defines the Reeb flow. In particular, flow lines are given by $z_{i}(t)=e^{it}z_{i}(0)$, i.e., the orbits of the Reeb vector field are the fibers of the Hopf fibration.

\section{Proof of the main result}
\label{sec:proof}

Here is the sketch of the argument proving Theorem \ref{theorem:main_result}. First observe that if $p,q$ are points so that $q$ does not lie on the flow line through $p$, then we can find open sets $U_{p},U_{q}$ so that \emph{no flow line passes through $U_{p}$ and $U_{q}$}. Indeed, this follows from the Hausdorffness of the space of flow lines $CP^{n-1}=S^{2n-1}/S^{1}$ (this is a rather special property of the standard contact form).

Introduce the following notation: given a contactomorphism $\varphi$ with scaling factor $g$, let $\Sigma_{\varphi}$ be the set $\set{g=0}$. The existence of a translated point implies the existence of a Reeb flow line joining $\Sigma_{\varphi}$ and $\varphi(\Sigma_{\varphi})$.

Our strategy is simple: construct $\varphi$ so that $\Sigma_{\varphi}\subset U_{p}$ while $\varphi(\Sigma_{\varphi})
\subset U_{q}$; clearly $\varphi$ will have no translated points.

\begin{defn}\label{defn:focal}
  For the purposes of the argument, let us say that a contactomorphism of a compact manifold $\varphi:Y\to Y$ has the \emph{focal property} for $(p,q)$ provided the following hold:
  \begin{enumerate}
  \item $p,q$ are fixed points of $\varphi$, $(\varphi^{*}\alpha)_{p} > \alpha_{p}$, and $(\varphi^{*}\alpha)_{q}<\alpha_{q}$,
  \item $q$ has arbitrarily small neighbourhoods $U$ satisfying $\varphi(U)\subset U$ (attracting).
  \item denoting $\varphi_{n}=\varphi\circ\dots\circ \varphi$, if $z\ne p$, then $\lim_{n\to\infty}\varphi_{n}(z)=q$.
  \end{enumerate}
  Notice that (iii) implies that any compact set $K$ disjoint from $p$ will eventually be mapped into arbitrarily small open sets around $q$.
\end{defn}
\begin{lemma}\label{lemma:technical_1}
  Let $\varphi:Y\to Y$ be a contactomorphism satisfying the focal property for $(p,q)$. Denote by $\Sigma_{n}$ the set of points $z$ so that $((\varphi_{n})^{*}\alpha)_{z}=\alpha_{z}$. Then $$\lim_{n\to\infty}\mathrm{dist}(\Sigma_{n},p)+\mathrm{dist}(\varphi_{n}(\Sigma_{n}),q)=0,$$ i.e., $\Sigma_{n}$ eventually enters arbitrarily small neighbourhoods of $p$ and $\varphi_{n}(\Sigma_{n})$ eventually enters arbitarily small neighbourhoods of $q$.
\end{lemma}
\begin{proof}
  Bear in mind that $Y$ is assumed compact. Let $\Sigma_{n}=\set{z\in Y:(\varphi_{n}^{*}\alpha)_{z}=\alpha_{z}}$, and let $U_{p},U_{q}$ be arbitrary (small) open sets around $p,q$ respectively. It suffices to show that $\Sigma_{n}\subset U_{p}$ and $\varphi_{n}(\Sigma_{n})\subset U_{q}$ for $n$ sufficiently large. Let $g$ be the scaling factor for $\varphi$, i.e., $\varphi^{*}\alpha=e^{g}\alpha$, and let $g_{n}$ be the scaling factor for $\varphi_{n}:=\varphi\circ\dots\circ \varphi$. The focal property (i) implies that $g(p)>0$ and $g(q)<0$. A straightforward computation establishes that:
  \begin{equation}
    \label{eq:scaling_iterate}
    g_{n}=g+g\circ \varphi+\dots+g\circ \varphi_{n-1}.
  \end{equation}
  It is clear that $g$ is a bounded function, so pick some $M>\sup_{z}\abs{g(z)}$.

  Shrinking $U_{p},U_{q}$ if necessary, and using the focal properties for $\varphi$, we may suppose that:
  \begin{enumerate}[label=(\alph*)]
  \item $\varphi(U_{q})\subset U_{q}$,
  \item $g>\delta$ on $U_{p}$ and $g<-\delta$ on $U_{q}$ for some $\delta>0$,
  \item $\varphi_{N}(Y-U_{p})\subset U_{q}$ for some $N\in \mathbb{N}$.
  \end{enumerate}
  We will refer to the constants $N,M,\delta$ in the subsequent arguments.

  Suppose that $z\in \Sigma_{n}$, and $z\not\in U_{p}$. Then $\varphi_{k}(z)\in U_{q}$ for all $k\ge N$. In particular, $g(\varphi_{k}(z))<-\delta$ for $k\ge N$. We then estimate:
  \begin{equation*}
    0=g_{n}(z)=\underbrace{g(z)+\dots+g(\varphi_{N-1}(z))}_{N\text{ terms}}+\underbrace{g(\varphi_{N}(z))+\dots+g(\varphi_{n-1}(z))}_{n-N\text{ terms}}<NM-(n-N)\delta.
  \end{equation*}
  Thus for $n$ sufficiently large we have a contradiction. Thus $\Sigma_{n}\subset U_{p}$, eventually.

  On the other hand,\footnote{To prove that $\varphi_{n}(\Sigma_{n})\in U_{q}$ for $n$ sufficiently large, we can also observe that $\varphi^{-1}$ has the focal property for $(q,p)$, and that:
    \begin{equation*}
      ((\varphi_{n}^{-1})^{*}\alpha)_{\varphi_{n}(z)}=\alpha_{\varphi_{n}(z)}\iff \alpha_{z}=(\varphi_{n}^{*}\alpha)_{z},
    \end{equation*}
    i.e., $\varphi_{n}(\Sigma_{n})=\set{\varphi_{n}^{-1}\alpha=\alpha}$. Thus the second part of the proof is a consequence the first half. It is not hard to show that $\varphi^{-1}$ has the focal property for $q,p$, although one needs to come up with a slightly clever choice of neighbourhood basis at $p$ to establish property (ii) for the inverse.
  } suppose that $z\in \Sigma_{n}$ but $\varphi_{n}(z)\not\in U_{q}$. Thanks to the previous part, we may assume that $z\in U_{p}$. Let $k$ be the smallest integer so that $\varphi_{k}(z)\not\in U_{p}$. We then estimate:
  \begin{equation*}
    0=\underbrace{g(z)+g(\varphi(z))+\dots+g(\varphi_{k-1}(z))}_{k\text{ terms}}+\underbrace{g(\varphi_{k}(z))+\dots+g(\varphi_{n-1}(z))}_{n-k\text{ terms}}>k\delta-(n-k)M.
  \end{equation*}
  Rearranging yields:
  \begin{equation*}
    k<\frac{n}{\delta/M+1}
  \end{equation*}
  For $n$ sufficiently large (depending only on $\delta,M$, and not on $z$), we have
  \begin{equation*}
    \frac{n}{\delta/M+1}\le n-N,
  \end{equation*}
  Thus $k<n-N$, and hence $\varphi_{k}(z)\not\in U_{p}$ implies $\varphi_{k+N}(z)\in U_{q}$ and hence $\varphi_{n}(z)\in U_{q}$. Thus, since $z$ was arbitrary, $\varphi_{n}(\Sigma_{n})\subset U_{q}$, as desired.
\end{proof}
\begin{cor}\label{cor:cor}
  If there exists a contactomorphism $\varphi:S^{2n-1}\to S^{2n-1}$ which has the focal property for $(p,q)$, and $q$ does not lie on the Hopf circle through $p$, then a sufficiently large iterate of $\varphi$ will have no translated points.
\end{cor}
\begin{proof}
  As explained above, we can find open sets $U_{p},U_{q}$ around $p,q$, respectively, so that no Reeb orbit passes through $U_{p}$ and $U_{q}$. Since $\varphi$ has the focal property, Lemma \ref{lemma:technical_1} guarantees that eventually $\Sigma_{n}\subset U_{p}$ and $\varphi_{n}(\Sigma_{n})\subset U_{q}$. Thus there are no flow lines joining $\Sigma_{n}$ to $\varphi_{n}(\Sigma_{n})$, so the iterate $\varphi_{n}$ has no translated points, as desired.
\end{proof}
Therefore, in order to prove Theorem \ref{theorem:main_result}, it suffices to construct a contactomorphism isotopic to the identity which satisfies the focal property for $p,q$ with $q$ disjoint from the Reeb orbit through $p$. We perform this construction in the next section.

\section{Constructing contactomorphisms with the focal property}
\label{sec:construction}
First we observe that the focal property is preserved under conjugation:
\begin{lemma}
  If $\varphi$ has the focal property for $(p,q)$, and $\sigma$ is any contactomorphism, then $\sigma\circ\varphi\circ\sigma^{-1}$ has the focal property for $(\sigma(p),\sigma(q))$. 
\end{lemma}
\begin{proof}
  Let $h$ be the scaling factor for $\sigma$ and $g$ the scaling factor for $\varphi$. Then, the scaling factor of $\sigma\circ \varphi\circ \sigma^{-1}$ equals $h\circ \varphi\circ \sigma^{-1}+g\circ \sigma^{-1}-h\circ \sigma^{-1}$. Using this formula, and the fact that $p,q$ are fixed points for $\varphi$, focal property (i) with $(\sigma(p),\sigma(q))$ is easily established for the conjugated contactomorphism. The focal properties (ii) and (iii) are straightforward to check, and are left to the reader.
\end{proof}

Now let $p,q$ be two points on $S^{2n-1}$, so that $q$ is not on the Reeb flow line through $p$ (this forces $n>1$). Since the contactomorphism group acts 2-transitively, the existence of a focal contactomorphism for any other pair $(P,Q)$ implies the existence of a focal contactomorphism for $(p,q)$. The following explicit formula proves the existence of a focal contactomorphism for a specific pair $(P,Q)$.
\begin{prop}[see Remark 8.2 in \cite{ekp}]\label{prop:explicit}
  Let $a\in (0,1)$ and for $(z_{1},\dots,z_{n})\in \C^{n}$ consider the mapping:
  \begin{equation*}
    \varphi(z)=(\frac{(1+a^{2})z_{1}+(1-a^{2})}{(1-a^{2})z_{1}+(1+a^{2})},\frac{2a z_{2}}{(1-a^{2})z_{1}+(1+a^{2})},\dots,\frac{2a z_{n}}{(1-a^{2})z_{1}+(1+a^{2})}).
  \end{equation*}
  Then $\varphi(S^{2n-1})\subset S^{2n-1}$, and $\varphi$ induces a contactomorphism $S^{2n-1}\to S^{2n-1}$, isotopic to the identity, which is focal for $P=(-1,0,\dots,0)$ and $Q=(1,0,\dots,0)$.
\end{prop}
\begin{proof}
  The $(n+1)\times(n+1)$ matrix:
  \begin{equation*}
    M_{a}=\frac{1}{2a}\begin{dmatrix}
      {1+a^{2}}&{1-a^{2}}&{0}\\
      {1-a^{2}}&{1+a^{2}}&{0}\\
      {0}&{0}&{2a1_{(n-1)\times (n-1)}}
    \end{dmatrix}
  \end{equation*}
  acts on $\C^{n+1}$ and preserves the quadratic form $q=-u_{0}\bar{u}_{0}+u_{1}\bar{u}_{1}+\dots+u_{n}\bar{u}_{n}$, i.e., $M_{a}$ lies in the group $U(n,1)$. Projectivizing via $z_{i}=u_{i}/u_{0}$, we see that the quotient group $PU(n,1)=U(n,1)/S^{1}$ acts by biholomorphisms of the unit ball in $\C^{n}$ (where $q<0$) and extends smoothly to the unit sphere (where $q=0$). Thus $PU(n,1)$ acts on $S^{2n-1}$ by contactomorphisms, since the contact distribution is the distribution of complex tangencies $TS^{2n-1}\cap JTS^{2n-1}$. Our formula for $\varphi$ is given by the action of $M_{a}$, and hence $\varphi$ is a contactomorphism. Moreover $U(n,1)$ (and hence $PU(n,1)$) is a connected Lie group,\footnote{Proof of connectedness: by an explicit argument, one can deform the columns of any $U(n,1)$ matrix to ensure the first column is $e_{0}$. Then the other columns form a unitary basis for $\set{0}\times \C^{n}$. Thus everything in $U(n,1)$ can be joined to an embedded copy of $U(n)$, which is connected.} and so $\varphi$ is isotopic to the identity.\footnote{One can also see directly that $\varphi\to \id$ as $a\to 1$.}
  
  It remains only to verify the focal properties. We see that $\varphi(z)=z$ if and only if $z_{1}=\pm 1$ (hence $z_{2}=\dots=z_{n}=0$), and so $P,Q$ are the only fixed points of $\varphi$. 

  The tangent space to $S^{2n-1}$ at $P,Q$ is equal to $i\R\oplus \C^{n-1}$, and, if $1_{i\R},1_{\C^{n-1}}$ denote the projections onto these subspaces (which are the characteristic line and the contact hyperplane, respectively), we can write the derivatives as:\footnote{This formula for the derivative is related to the rescaling contactomorphism $(x,y,z)\mapsto (cx,cy,c^{2}z)$. See \cite[pp.\ 1743]{ekp} for further details.}
  \begin{equation*}
    \begin{aligned}
      \d\varphi_{P}&=c_{P}^{2}1_{i\R}+c_{P}1_{\C^{n-1}}\text{ where }c_{P}=1/a>1,\\
      \d\varphi_{Q}&=c_{Q}^{2}1_{i\R}+c_{Q}1_{\C^{n-1}}\text{ where }c_{Q}=a<1.
    \end{aligned}
  \end{equation*}
  The focal property (i) follows immediately. Some basic calculus also establishes focal property (ii) (i.e., by comparing the map with its derivative). Finally, recalling that focal property (iii) states $z\ne P\implies \lim_{n}\varphi_{n}(z)=Q$, we argue as follows: the sequence $\varphi_{n}(z)$ must have its limit points contained in the fixed point set $\set{P,Q}$. Moreover, if $\varphi_{n}(z)$ has $Q$ as a limit point, then, by the attracting property, $\varphi_{n}(z)$ must converge to $Q$. Therefore if $\varphi_{n}(z)$ does not converge to $Q$ it must converge to $P$. However, since the derivative at $P$ is expanding, it is clear that $\varphi_{n}(z)\ne P$ cannot converge to $P$. This completes the proof
\end{proof}

Thus we conclude the existence of a focal contactomorphism for the chosen pair $(p,q)$; simply take the explicit formula from Proposition \ref{prop:explicit}, and conjugate by a contactomorphism\footnote{The explicit formula and the conjugation argument were explained to me by Egor Shelukhin} $\sigma$ which takes $P=(-1,0,\dots,0)$ to $p$ and $Q=(1,0,\dots,0)$ to $q$. Applying Corollary \ref{cor:cor} then completes the proof of Theorem \ref{theorem:main_result}.

\begin{remark}
  It is clear that our construction is related to the \emph{contact (non)squeezing problem} for domains in the standard sphere. Indeed, the existence of contactomorphisms with the focal property implies that large domains can be squeezed inside of arbitrarily small domains. In \cite{uljarevic}, the author proves a non-squeezing property holds for domains in certain non-standard contact spheres, namely the \emph{Ustilovsky spheres}. As part of the argument, the author shows that the Ustilovsky spheres admit Liouville fillings with infinite dimensional symplectic homology, and hence Sandon's conjecture holds in this case by the results of \cite{merry_ulja}.
\end{remark}

% \begin{remark}[citations to add]
%    \cite{albers_kang}, \cite{theret_rotation}, \cite{allais_lens}, \cite{allais_zoll}, \cite{gkps}, \cite{givental_quasimorphism}.
% \end{remark}
\bibliography{citations}
\bibliographystyle{alpha}
\end{document}